\documentclass[11pt]{amsart}
\usepackage[T1]{fontenc}
\usepackage[ansinew]{inputenc}
\usepackage{amssymb,amsopn,amsmath,amsthm,authblk,hyperref, cleveref,lipsum,scrextend, mathrsfs,array,graphicx,xcolor,bbm,enumerate,wasysym,yfonts}
\usepackage{stmaryrd}
\usepackage{subcaption}
\usepackage{verbatim}
\usepackage[all]{xy}
\usepackage{tikz}
\usepackage[displaymath, mathlines]{lineno}
\usepackage{fullpage}
\usepackage[foot]{amsaddr}

\newtheorem{thm}{Theorem}[section]
\newtheorem{lemma}[thm]{Lemma}
\newtheorem{rem}[thm]{Remark}

\newtheorem{prop}[thm]{Proposition}
\newtheorem{cor}[thm]{Corollary}

\newcommand{\tmb}{Tree-decorated map with a simple boundary }

\def\cro#1{\llbracket#1\rrbracket}

\newcommand{\sm}{\mathfrak{sm}}

\newcommand{\cc}{\mathfrak{c}}
\newcommand{\mm}{\mathfrak{m}}

\renewcommand{\tt}{\mathfrak{t}}
\newcommand{\ff}{\mathfrak{f}}

\newcommand{\R}{\mathbb R}

\newcommand{\N}{\mathbb N}

\renewcommand{\S}{\mathbf S}

\newcommand{\BBB}{\mathcal B}
\newcommand{\SSS}{\mathcal S}
\newcommand{\bbb}{b}
\newcommand{\sss}{s}

\newcommand{\C}{\mathcal{C}}

\newcommand{\F}[1]{\mathsf{F}_{#1}}

\newcommand{\T}[1]{\mathsf{T}_{#1}}
\newcommand{\PPT}[1]{\mathsf{T}_{#1}}
\newcommand{\M}{\mathsf{M}}
\newcommand{\MM}[1]{\mathsf{M}_{#1}}
\newcommand{\DSTM}[1]{\mathring{\mathsf{M}}^{\mathsf{ST}}_{#1}}
\newcommand{\DSTMA}[1]{\mathsf{M}^{\mathsf{ST}}_{#1}}

\newcommand{\DTM}[2]{\mathring{\mathsf{M}}^{\mathsf{T},#2}_{#1}}
\newcommand{\DTMA}[2]{\mathsf{M}^{\mathsf{T},#2}_{#1}}
\newcommand{\TMSB}[3]{\mathsf{B}^{\mathsf{T},#3}_{#1,#2}}
\newcommand{\MMSB}[2]{\mathsf{MSB}_{#1,#2}}
\newcommand{\DFM}[2]{\mathring{\mathsf{M}}^{\mathsf{F},#2}_{#1}}
\newcommand{\DFMA}[2]{\mathsf{M}^{\mathsf{F},#2}_{#1}}
\newcommand{\DFMAG}[2]{\tilde{\mathsf{M}}^{\mathsf{F},#2}_{#1}}

\newcommand{\DSFMA}[2]{\mathsf{M}^{\mathsf{SF},#2}_{#1}}

\newcommand{\MB}[2]{\mathsf{B}_{#1,#2}}
\newcommand{\MSB}[2]{\mathsf{SB}_{#1,#2}}

\newcommand{\PC}[2]{\mathring{\mathsf{M}}^{\mathsf{C,#2}}_{#1}}

\newcommand{\PCE}[2]{\mathring{\mathsf{M}}^{\mathsf{C,#2}}_{E=#1}}
\newcommand{\PCAE}[2]{\mathsf{M}^{\mathsf{C},#2}_{E=#1}}

\makeatletter
\g@addto@macro{\endabstract}{\@setabstract}
\newcommand{\authorfootnotes}{\renewcommand\thefootnote{\@fnsymbol\c@footnote}}
\makeatother

\let \epsilon \varepsilon

\numberwithin{equation}{section}

\crefformat{footnote}{#2\footnotemark[#1]#3}

	\title{Tree-decorated planar maps}
	\author{Luis Fredes$^*$}
	\address{$^*$Université de Bordeaux, LaBRI.}
	\author{Avelio Sepúlveda$^\dagger$}
	\address{$^\dagger$Univ Lyon, Universit\'e Claude Bernard Lyon 1, CNRS UMR 5208, Institut Camille Jordan, 69622 Villeurbanne, France.}
\begin{document}
	\maketitle
\begin{abstract}
We introduce the set of (non-spanning) tree-decorated planar maps, and show that they are in bijection with the Cartesian product between the set of trees and the set of maps with a simple boundary. As a consequence, we count the number of tree decorated triangulations and quadrangulations with a given amount of faces and for a given size of the tree. 
Finally, we generalise the bijection to study other types of decorated planar maps and obtain explicit counting formulas for them.  
\end{abstract}

\section{Introduction}

In this paper, we study the combinatorial properties of tree-decorated maps via the use of a simple bijection. Tree-decorated maps are a couple made of a (planar) map and a subtree. They are one-parameter families interpolating between planar maps, when the subtree has size 1, and the spanning-tree decorated map, when the subtree has the same number of vertices that the whole graph. 

Planar maps and spanning-tree decorated maps have been thoroughly studied, both in combinatorics and probability. Planar maps were introduced in \cite{Ed60} and afterwards have been thoroughly studied in many works from both a combinatorial (see for example \cite{Tut62,SCH98,BG}) and a probabilistic perspective (see for example \cite{MM,LeGall,Mier}). Spanning-tree decorated maps were first studied in \cite{Mul67}, where they obtain a simple closed formula to count them, which later was explained in \cite{WL72,CDV86,Ber07} through bijective methods. These bijections became key to the study of planar maps decorated by statistical physics models  \cite{She}.

\subsection{Motivation}
The main motivation for the introduction of this model is to try to understand the difference,  as metric spaces, between uniformly chosen planar maps and uniformly chosen spanning-tree decorated planar maps. 

\begin{figure}
	\includegraphics[width=0.9\textwidth]{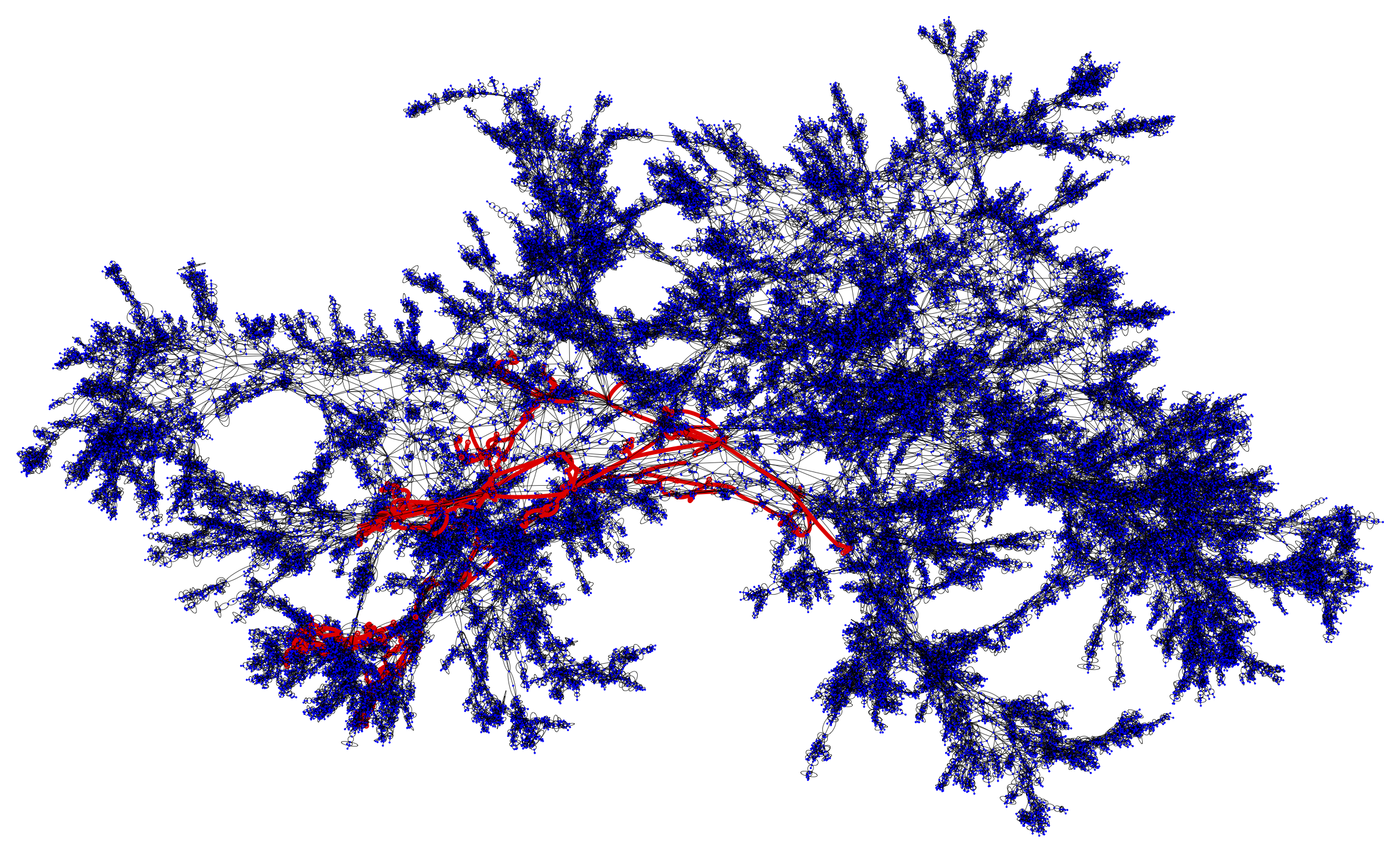}
		\caption{A simulation, based in our bijection, of a uniformly chosen tree decorated quadrangulation with $90000$ faces and with a tree of $500$ edges. The tree is in red and the image is close to an isometrically embedding of the map in $\R^3$.}
%	\includegraphics[width=0.9\textwidth]{Images/WQ30k50-1.png}
%	\caption{A simulation, based in our bijection, of a uniformly chosen tree decorated quadrangulation with $30000$ faces and with a tree of $50$ edges. The tree is in red and the embedding tries to isometrically embed the map in $\R^3$.}
\end{figure} 

The study of planar maps as metric spaces has been an active area of research in combinatorics and probabilities. In the case of a uniformly chosen planar quadrangulation, Cori-Vauquelin-Schaeffer (CVS) bijection \cite{SCH98} has been used to understand many of its asymptotic properties: the distances scale like the number of edges to the power $1/4$ and there is an explicit limiting metric space called Brownian map \cite{MM,Mier, LeGall}. However, in the world of uniformly chosen spanning-trees decorated planar maps, we do not know much. We can only give bounds on the order of the diameter as a function of the number of edges \cite{GHS,DGExpB}. On the optimistic side, Walsh and Lehman's bijection \cite{WL72} shows that the spanning tree decorating the map is chosen uniformly over all trees with a certain size (number of vertices).

The main reason why it is difficult to understand distances in the context of spanning-tree decorated maps becomes clear when one compares it to the case of planar quadrangulations. The main tool used to study distances in these planar maps is the (CVS) bijection, which relates a planar map to a pair of trees. In this bijection, one of the trees encodes the distance to a marked point, thus by knowing this tree, one gets an immediate upper bound on the diameter of the map.  On the side of tree-decorated planar maps, we are not that lucky. Even though Walsh and Lehman's bijections \cite{WL72,Ber07} relate them to a pair of trees, it is not possible to extract any information about distances in the original graph from just looking at either of them.

At this point, let us make a remark from the point of view of conformal field theory, where these two models are not expected to look the same. Uniformly chosen planar maps are model associated to central charge equal to 0 ($c=0$), while spanning-tree decorated maps have an associated central charge of $-2$ ($c=-2$) \cite{JS}. Thus, two objects, in the world of tree-decorated maps, have two-different central charges, meaning that theirs conformal properties are interacting. Trying to understand how this interaction works in the limit is the main interest of a follow up paper \cite{FS}.

\subsection{Results}
Let us, now, present the main results of the paper. For detailed definitions, we refer the reader to Section \ref{s.preliminaries}.

The main result of this paper is an explicit bijection between tree-decorated maps and maps with a simple boundary quotiented by a tree.
\begin{thm}\label{theo1}
	 The set of tree-decorated maps with $f$ faces, with a tree of $m$ edges and a root-edge on the tree, is in bijection with the Cartesian product of trees with $m$ edges and maps with a simple boundary of size $2m$ and $f$ interior faces. 
\end{thm}

The bijection can be summarise as followed: the tree is kept and the map with a boundary is produced by a duplication of the edges of the tree. We call this direction of the bijection ungluing, and its inverse gluing. Note that to gluing consists only in identifying the boundary of the map using the equivalence relationship generated by the tree. Additionally, let us note that as the decorating tree is kept without any changes, in a uniformly chosen tree-decorated map, the law of the tree is uniform among all the trees of a given size (\Cref{c.uniform}). This is a generalisation of the result for spanning-trees decorated maps. Furthermore, this also shows that the bijection also works when restricted to maps decorated by a special type of tree.

The main contribution of \Cref{theo1} is to connect the study of tree-decorated maps to that of maps with a boundary. In the spirit of the motivation, note that maps with boundary are also an interpolation model, as when the boundary is of size $2$, we recover planar maps, and, when the boundary has twice the number of edges, we recover planar trees.  This effect was explored, for uniformly chosen maps with a boundary in \cite{BetQ}, and we hope that the bijection allows us to transfer the transition from map to a tree to that from a map to a spanning tree-decorated map.

Important consequences of \Cref{theo1} are counting formulas for some subsets of tree-decorated maps. A close look at the bijection of \Cref{theo1} shows that the ungluing  procedure changes only the  newly created face. Thus, it is possible to obtain counting formulae for tree-decorated (and  spanning-tree decorated) $q$-angulations. To obtain these results, we need to count the maps with a simple boundary and use a re-rooting argument. Luckily, we can find these countings in \cite{Kr07} for triangulations and in \cite{BG09} for quadrangulations.

\begin{cor}\label{c.counting}
	For $m\leq f/2+1$, the number of tree-decorated triangulations where the tree has $m$ edges and the map has $f$ faces is
	\begin{equation}
	2^{f-2m}\frac{(3f/2+m-2)!!}{(f/2-m+1)!(f/2+3m)!!}\frac{3f}{m+1}\binom{4m}{2m,m,m},
	\end{equation}
	where $n!!$ stands for the double factorial of $n$.
	
	Furthermore, for $m\leq f+1$, the amount of tree-decorated quadrangulations, where the tree has size $m$ and the map has $f$ faces
	\begin{equation}
	3^{f-m}\frac{(2f+m-1)!}{(f+2m)! (f-m+1)!} \frac{4f}{m+1}\binom{3m}{m,m,m}. 
	\end{equation}
\end{cor}

Note that triangulations, resp. quadrangulations, with $f$ faces have $f/2+2$ vertices, resp. $f+2$ vertices. Thus, one can also find the number of spanning tree-decorated triangulations and quadrangulations. These formulas appeared first in \cite{BM11}, and were based in Walsh and Lehman's bijection \cite{WL72}. Thus, these results give a new bijective proof of them.
\begin{cor}\label{cor:sptree}
	The number of spanning-tree decorated triangulations of $f$ faces and the root-edge could be placed in every possible oriented edge is
	\begin{equation}\label{e.stree tri}
	\frac{12f}{(f+4)(f+2)^2}\binom{2f}{f,f/2,f/2}.
	\end{equation}
	Furthermore, the cardinal of the set of spanning-tree decorated quadrangulation with $f$ faces and the root-edge could be placed in every possible oriented edge is
	\begin{equation}\label{e.stree quad}
	\frac{4f}{(f+1)^2(f+2)}\binom{3f}{f,f,f}.
	\end{equation}
\end{cor}
Let us remark that the counting formula for general maps with a given boundary and edges size exists and is not closed so we decided not to present it here. Nevertheless, we compute its generating function in Section \ref{ssec:sbound}. Furthermore, note that the spanning-tree decorated maps with a given number of faces is infinite, so as it is stated the bijection is not interesting in that case.  Nevertheless, given that the bijection does not change the vertices, edges and faces in the interior the same proof can be adapted to obtain the following result.
	\begin{thm}\label{theo2}
		The set of tree-decorated maps with $e+m$ edges decorated by a tree with $m$ edges and a root-edge on the tree, is in bijection with the Cartesian product of trees with $m$ edges and maps with a simple boundary of size $2m$ and $e$ interior edges. 
	\end{thm}

The main bijection is not only useful to obtain counting results, but we plan to use it in \cite{FS} to study distances in uniformly chosen tree decorated maps. The bijection becomes useful because distances in uniformly chosen maps with boundaries (simple or not) have been intensively studied  \cite{BetQ,BetM,GM} and are better understood than that of tree-decorated maps.

Let us, also, mention that the bijection presented in this work is simple enough so that decoration on the vertices can be carried between the two objects. This may allow, in the future, to understand what happens when the tree-decorated map is not chosen uniformly but weighted by a statistical physic law. This probability laws, have been the object of great interest in the statistical physics community, especially after the introduction of the so-called `Hamburger-Cheeseburger'-bijection \cite{She}.

Most of the other results of this article consist in using the main idea of the bijection of Theorem \ref{theo1} to produce bijections between other combinatorial objects. In this flavour, we study forest-decorated planar maps (see \Cref{cor2}) and tree-decorated planar maps with a simple boundary (see \Cref{cor1}). The latter one allows us to make the gluing procedure in a progressive dynamical way (see \Cref{prop:nonsimple}).

On another note, we also explain what happens when one tries to use our bijection with maps whose boundary is not necessarily simple. In this case, the gluing does not produce maps, but what we call bubble-maps (see \Cref{sec:nonsimple}), which are maps embedded in a tree-like structure of spheres, and which are decorated by a specific type of circuit (see \Cref{prop:nonsimple}).

As a final remark, we would like to say that gluing maps along the boundary has already been done, but not when the decoration is a tree. Some gluings in the literature are, for example, self avoiding walks \cite{DK88,BetQ,GM2,CarCur} and loops \cite{BBG12}. 

	\subsection{Organisation of the paper}
The paper is organised as follows: we start with the preliminaries, where we carefully present the classical objects that are important for the paper together with the maps we are interested in. In \Cref{Bij}, we present all the bijections and its proofs. Finally, in \Cref{count}, we discuss the counting formulas we obtain from the bijections.

\subsection{Acknowledgments}
We would like to thank Jérémie Bettinelli, Mireille Bousquet-Mélou, Éric Fusy, Vincent Jugé, Jean François Marckert, Gregory Miermont and Neil Sloane for interesting and fruitful discussions. The work of this paper started and was partially done during various trips of both authors financed by ERC grant LiKo 676999. Also, we would like to thank  Núcleo Milenio Stochastic models of complex and disordered system for inviting us to Chile, where part of this work was done, and to "Mobilité junior" of LaBRI who financed part of the trip of L.F. The research of L.F. has been partially supported by ANR GRAAL (ANR-14-CE25-0014). The research of A.S. is supported by ERC grant LiKo 676999.

\section{Preliminaries}\label{s.preliminaries}
\subsection{Elementary definitions}\label{s.elementary def}
In this section, we present the elementary concepts that appear in this work. As the main result of the paper is a self-contained bijection, the counting results are only stated when needed. For an introduction to planar maps, we recommend, for example, to see \cite{GJ04, FS09} and \cite{BM11}.

A \textsc{rooted planar map}, or \textsc{map} for short, is a finite connected graph embedded in the sphere that has a marked oriented edge. We consider that two maps are the same if there is an  homeomorphism between them that preserves the orientation (i.e. respecting a cyclic order around every vertex). We call \textsc{root edge} the marked oriented edge and \textsc{root vertex} its starting point. 

The \textsc{faces} of a map are the connected components of the complement of the edges in the embedding. The \textsc{degree} of a face is the number of oriented edges for which the face lies at its left. We call \textsc{root face} the face that is to the left of the root, following the sense of the root edge. A \textsc{$q$-angulation} is a map where each of the faces has degree $q$. We set $\MM{f}$ to be the set of all maps on $f$ faces. Finally, we call a \textsc{planted plane tree}, or \textsc{tree} for short, to a map with a single face and we denote the set of planted plane trees with $m$ edges as $\PPT{m}$.

To understand the properties of $\PPT{m}$, it is useful to codify trees using walks. In the literature, there are several of these codings, see for example Section 1 of \cite{LG05}. In this paper, we are interested in the most elementary of them: the contour function. We present this bijection briefly, for a more detailed description we refer to Section 1.1 of \cite{LG05} and Section 2 of \cite{Ber07}.

To introduce the contour function, let us first note that a tree has an intrinsic way of visiting all of its oriented-edges. This visit can be represented by a walker that starts from the root vertex and follows the direction of the root edge touching the tree with his left hand\footnote{Note that, in the literature, the walker usually walks following its right hand. In this work, the left hand convention makes some statements easier.}
as long as it walks. The walker, then, continues until it returns to the root edge. Note that this walk visits every oriented edge only once. Now, we define the contour function $C : \cro{0,2m} \to \N$ as giving, at time $n$, the distance to the root vertex (height) of the vertex visited at time $n$ by the walker.

The inverse of this bijection given by the contour function is explicit. We say that a function $C:\cro{0,2m}\to \N$ is a contour function if $C(0)=C(2m)=0$ and its increments are $\pm 1$, i.e., $C$ is a Dyck path. We can construct a plane tree from a contour function by saying that two points $n_1,n_2\in \cro{0,2m}$ are equivalent if \begin{equation}
\label{e.EquivalenceC}
C(n_1)=C(n_2)=\inf_{n\in \cro{n_1,n_2}} C(n).
\end{equation} The vertices of the tree can be recovered from the equivalence classes of the relation and  the edges can be recovered as follows: two vertices have an edge between them if they are the equivalent class of two elements that are exactly at distance $1$ in $\cro{0,2m}$. Note that each edge comes exactly from two steps of the walk, one going up and the other one coming down.
\begin{figure}[!h]
	\centering
	\includegraphics[scale=1]{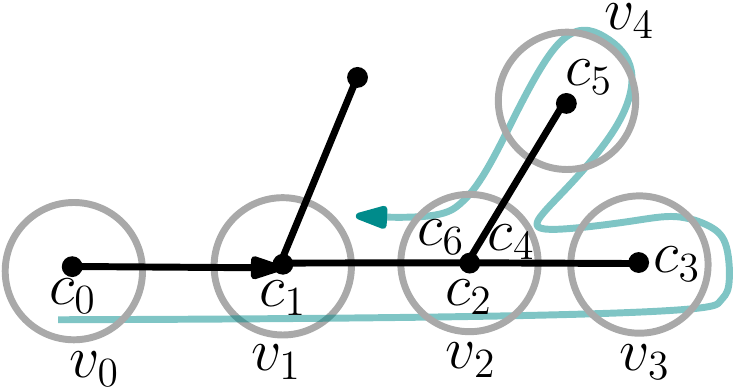}
	% \vspace{-0.6cm}
	\caption{{\Small Tree with part of the contour in green, the corners are numbered as $c_i$ and are shown in gray. All the corners belonging to the same circle belongs to the same equivalence class associated to a vertex, for example $v_2 = [c_2]_c= [c_4]_c= [c_6]_c$. }}
	\label{conttree}
\end{figure} 

We say that $\ff$ is a \textsc{forest with $r$ trees}, if $\ff= (\tt_1,\;\dots,\;\tt_r)$, where $\tt_i$ is a tree for all $i\in \{ 1, 2, \dots, r\}$. We define $\F{m_1,m_2,\dots,m_r}$ as the set of all forest on $r$ trees with prescribed sizes $m_1,\; m_2,\dots,\;m_r$, meaning that for every $i\in \{ 1, 2, \dots, r\}$, $\tt_i$ has $m_i$ edges. 
%We also define $\FF{m,p}$ to be the set of forest with $p$ trees, where the total number of edges is $m-p$, i.e. the union of $\F{m_1,m_2,\dots,m_p}$ for $\sum_{i=1}^p m_i = m-p$.

We define a \textsc{non-self crossing circuit} as a sequence of directed edges $(e_i)_{i=0}^{l-1}$, for some $l\in \N$, embedded in the plane, such that the head of $e_i$ is the tail of $e_{i+1\mod l}$ for all $i\in \cro{0,l-1}$, and that it does not cross itself. The last part means that for every vertex in the circuit, we do not find the pattern $e_{i},e_{j+1},e_{i+1},e_{j}$ in cyclical order, where $i,j \in \cro{0,l-1}$ and the sum in the indexes is modulo $l$.

A \textsc{map with a boundary} is a map that has one special face, which has arbitrary degree, however, we will work only with those having arbitrary even degree. The special face is denoted \textsc{external face} and all the others are called \textsc{internal faces}. The oriented edges incident to the external face form the boundary of the map and the root edge must belong to the boundary. Recall that an oriented edge is incident to a face if the face lies at the left of the oriented edge, this implies that the external face is the root face. We call $\MB{f}{m}$ the set of maps with a boundary with $f$ internal faces and with a boundary of size $2m$.

Note that oriented edges in the boundary of a map with a boundary have a canonical label in $\cro{0,2m}$. This label is obtained from the amount of steps that a walker, who starts from the root edge and follows the boundary of the root face, takes to arrive to a given edge. An interesting case is when the boundary of the map is simple, i.e., it is simple as a curve. In this case, the labels of the edges gives rise to a labelling of the vertices of the boundary. To separate them from the non-simple ones, we denote the set of maps with a simple boundary in $\MB{f}{m}$ as $\MSB{f}{m}$.

Let us also introduce another type of boundary. A boundary is called \textsc{bridgeless} if the walk described above never goes twice along the same edge. Let us give an alternative definition. An edge is called a bridge if its suppression disconnects the map. Thus, a boundary is said to be bridgeless if it does not contain any bridges.

We can also work with more than one boundary. Let $r\in \N$ and define a \textsc{map with $r$-boundaries} as a map with multiple special faces, called the boundary faces, which are pairwise vertex-disjoint and have simple boundaries. Furthermore, this map is multi-rooted, that is to say, every boundary faces has a root and the roots are labelled from $1$ to $r$, where the root labelled 1 coincides with the root of the map. We called $\MMSB{f}{m_1,\dots,m_r}$ the set of all multi-rooted maps with boundaries on $f$ internal faces and with boundaries of size $2m_1,\dots,2m_r$ respectively for the labels of the boundaries.

A map $\mm_1$ is said to be a submap of $\mm_2$ (with the notation $\mm_1\subset_M \mm_2$) if $\mm_1$ can be obtained from $\mm_2$ by suppressing edges and vertices. This definition implies that $\mm_1$ respects the cyclic order of $\mm_2$ in the vertices and edges remaining. A decorated map is a map with an especial submap, i.e. it is a pair $(\mm, \sm(\mm))$ with $\mm \in \MM{f}$ for some $f\in \N$ and $\sm(\mm)\subset_M \mm$.

\subsection{Objects of interest and notation}\label{s.objectsinterest}
We present, now, the main objects of study of the paper. In particular, we rigorously define the tree-decorated map.

Take a map $\mm \in \MM{f}$ and define $\T{m}(\mm)$ to be the set of (unrooted) trees $\tt \subset_M \mm$ with $m$ edges. A \textsc{tree-decorated map} is a couple $(\mm,\tt)$ with $\mm \in \M_f$ and $\tt\in \T{m}(\mm)$. We denote  $\DTMA{f}{m}$ the set of all of these pairs. Furthermore, we denote $\DTM{f}{m}$ the set of all tree-decorated maps where the root of the tree is in the map. Note that one can do similar definitions when the map $\mm$ is a $q$-angulation.The set of tree-decorated triangulations and quadrangulations is counted in \Cref{c.counting} and that of spanning-tree decorated triangulations and quadrangulations is counted in \Cref{cor:sptree}.
		
 Let us now consider a map with a simple boundary $\mm^b \in \MSB{f}{m_1}$ and define $\T{m_2}^b(\mm)$ to be the set of trees $\tt \subset_M \mm^b$ with $m_2$ edges containing a unique vertex on the boundary, the root vertex. A \textsc{tree-decorated map with a simple boundary} is a pair $(\mm^b,\tt)$ with $\mm^b\in \MSB{f}{m_1}$ and $\tt\in\T{m_2}^b (\mm^b)$. We define the set of all these pairs as $\TMSB{f}{m_1}{m_2}$. The set of all tree-decorated triangulations with a simple boundary is counted by \eqref{e.tree simple boundary tri} and that of quadrangulations is counted by \eqref{e.tree simple boundary quad}.

To finish, we define an \textsc{$r$-forest-decorated map} as a map decorated on a forest with $r$ non vertex-intersecting (unrooted) trees. The set of all $r$-forest decorated maps on $f$ faces and with the trees of the forest of size $m_1,\dots, m_r$ respectively is denote by $\DFMA{f}{m_1,1\dots,m_r}$, note that here the order of $m_1,...,m_r$ is not important. Furthermore, we denote $\DFM{f}{m_1,m_2,\dots,m_r}$, the set of $r$-forest-decorated map where each tree is  rooted, each root is labelled with a different number in $\cro{1,r}$ and the root labelled $1$ coincides with the root edge of the map.  The counting of $r$-forest decorated triangulations is given in \eqref{e.forest tri} and that of quadrangulation is given in \eqref{e.forest quad}.

\section{Main Bijections}\label{Bij}
\subsection{The basic bijection} \label{s.b}

\subsubsection{From tree decorated maps  to maps with  simple boundary and trees} In the following paragraphs, we construct the first part of the bijection: the ungluing procedure. This procedure is represented by a function $\phi$ that takes a tree-decorated map in $\DTM{f}{m}$ and gives a tree in $\PPT{m}$ together with a map with a simple boundary in $\MSB{f}{m}$. Basically, in this procedure, the resulting tree is equal to the decorating tree and the map with a boundary is obtained by a duplication of the oriented edges of the tree in such a way that the newly appeared oriented edges all belong to the same face, see Figure \ref{dectoboundandtree}.
\begin{figure}[!h]
	\begin{subfigure}{0.31\textwidth}
		\centering
		\scalebox{1}[1]{\includegraphics[scale=0.4]{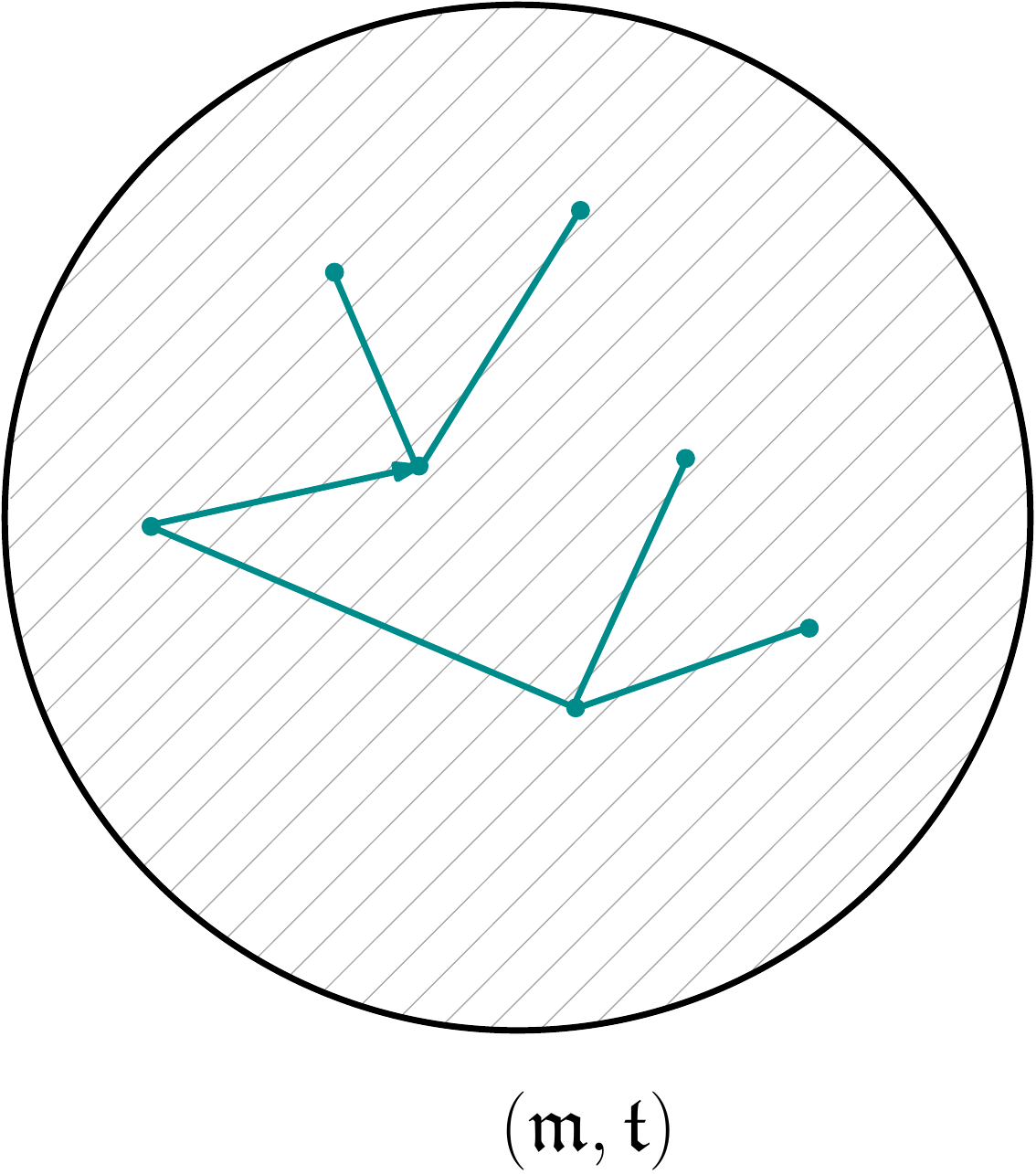}}
	\end{subfigure}
	\hspace{-0.25cm} $\substack{\phi\\\longrightarrow}$ \hspace{-0.75cm}
	\begin{subfigure}{0.63\textwidth}
		\centering
		\scalebox{1}[1]{\includegraphics[scale=0.4]{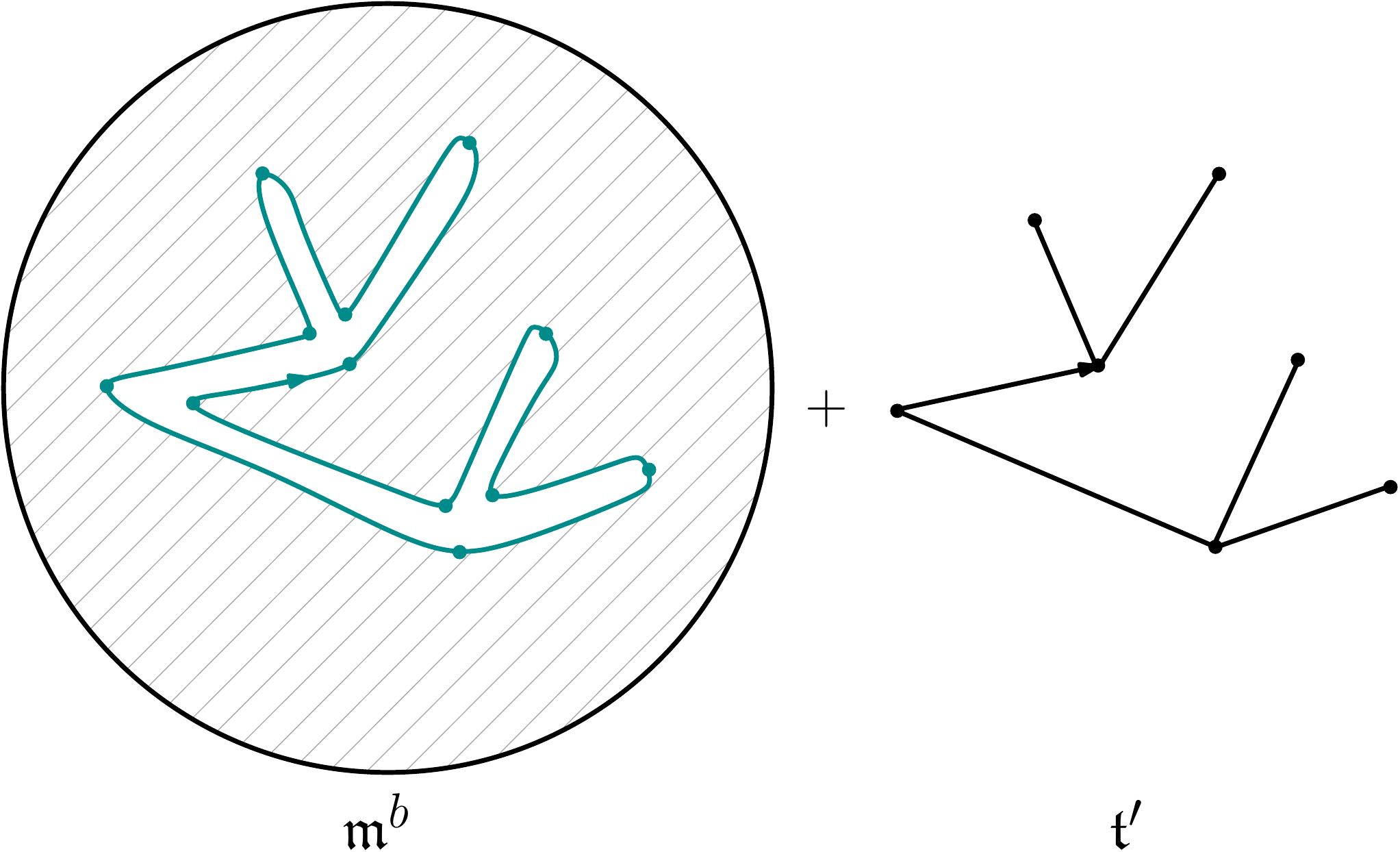}}
	\end{subfigure}
	% \vspace{-0.6cm}
	\caption{A simple sketch of the ungluing procedure.}
	\label{dectoboundandtree}
\end{figure}

 Consider a tree decorated map $(\mm,\tt)\in\DTM{f}{m}$, and denote $(\mm^b,\tt')$ as the (soon-to-be-constructed) image of $(\mm,\tt)$ under $\phi$. The tree $\tt'$ is taken equal to the tree $\tt$, in particular that the root edge of $\tt'$ is the same as that of $\tt$, i.e., the one of the tree decorated map. 
 
 The construction of the map with a boundary $\mm^{b}$ is not as straight-forward. We start by defining the notion of a corner of $\tt$ at a vertex $x$ as a pair of two consecutive (for the clockwise order) oriented edges, where the first one finishes at $x$ and the second one starts at $x$. Define $K$ as the set of corners of $\tt$. Note that an oriented edge that starts or ends in $x$ can be said to go to a corner $(e_1,e_2)$ of $x$ if it is between $e_1$ and $e_2$ for the clockwise order. 
 
 We can now define $\mm^b$. The set of vertices is the union between $K$ and the vertices of $\mm$ that are not in the tree $\tt$, i.e.,  $V(\mm)\backslash V(\tt)$.  To define the edges we just need to comeback to the original ones. Each oriented edge $e$ of $\mm$ whose vertices do not lie in $\tt$ is also an edge of $\mm^b$. Each oriented edge $e$ of $\mm$ that have at least one vertex in $\tt$, becomes an oriented edge that instead of the vertex in the tree has the corner in which the edge is incident. Finally, we add the oriented edges $(c_1,c_2)$ between two elements of $K$ if $(c_2,c_1)$ was already added before.

 To finish the definition of $\mm^b$ we are only missing the root edge and the embedding. The root edge of $\mm^b$ is the edge created from the root edge of $\mm$ by the procedure described. The embedding, is characterised by the cyclic orientation of the edges which is the same as that of $\mm$. In other words, in a corner the cyclic orientation is kept the same as in $\mm$, and this can be done because all edges in $\mm^b$ that belong to a corner can be identify with a unique edge in $\mm$.

\begin{figure}[!h]
	\begin{subfigure}{0.49\textwidth}
		\centering
		\includegraphics[scale=0.8,trim={1.5cm 4cm 3cm 2cm},clip]{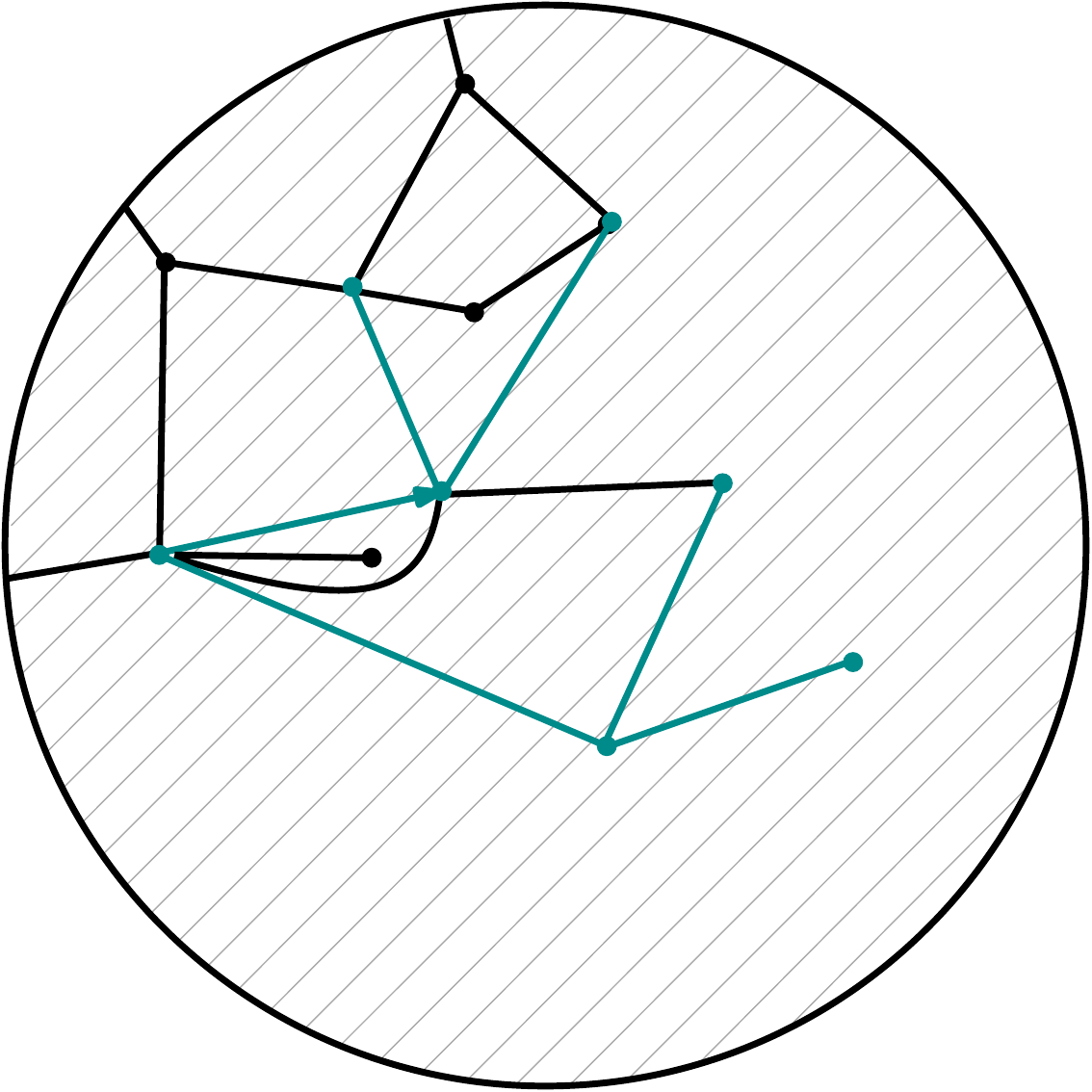}
	\end{subfigure}
	\hspace{-0.5cm}$\substack{\phi\\\longrightarrow}$\hspace{-0.5cm}
	\begin{subfigure}{0.49\textwidth}
		\centering
		\includegraphics[scale=0.8,trim={1.5cm 4cm 3cm 2cm},clip]{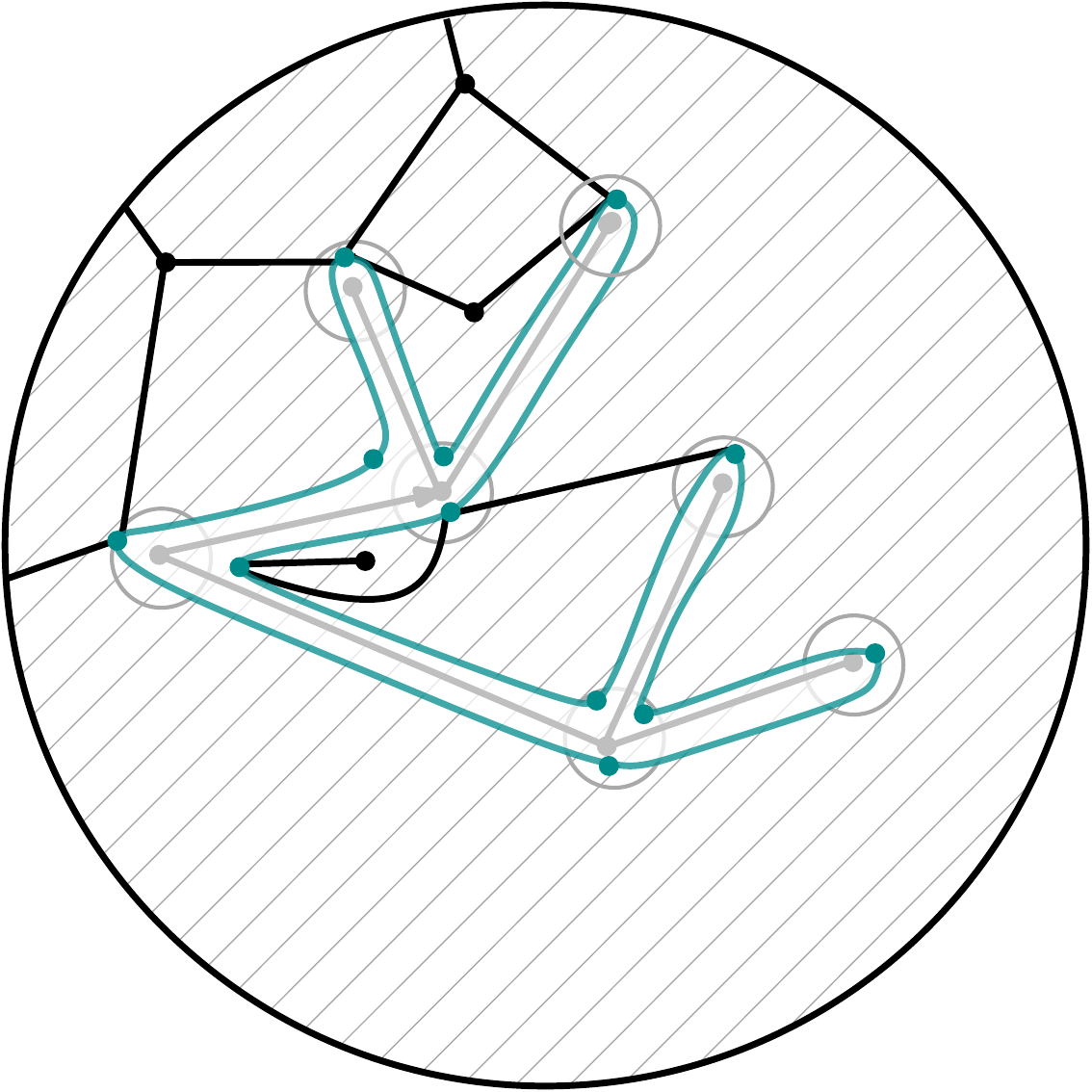}
	\end{subfigure}
	% \vspace{-0.6cm}
\caption{Local image of $\phi$ }
	\label{dectoboundandtree1}
\end{figure} 

Let us, now, show that the constructed map  actually belongs to $\MSB{f}{m}$. 
 \begin{lemma} \label{l.embedding1}
 	The constructed map $\mm^b$ belongs to $\MSB{f}{m}$.
 \end{lemma}
 \begin{proof}

We first need to note that the boundary of the root face has size $2m$. To do that we just note that the boundary of the root face is made by all the oriented edges that went between two corners, thus it has exactly size $2m$.  Furthermore, the boundary of the map $\mm^b$ is simple, otherwise the tree would have at least one identification of vertices generating a cycle. 

To finish, we just need to show that $\mm^b$ is a (planar) map, thus that it can be embedded in the sphere, $\S^2$. This comes directly from the fact that there is an orientation preserving isomorphism that goes from $\S^2\backslash \tt$ to $\S^2\cap\{(x,y,z):x>0\}$, and that can be extended to the boundary in a way that it sends prime ends of  $\tt$ in $\S^2\backslash \tt$ to points in $\S^2\cap\{x=0\}$ keeping the cyclical order between them \footnote{In this case it follows from the fact that $\S^2 \backslash \tt$ is simply connected (so one can use Riemman's theorem), together with results of the behaviour of conformal functions close to the boundary. See Chapter 2 of \cite{Pom} for definitions of prime ends and the main results in the boundary behaviour of conformal maps.}. Thus, one can embed $\mm^b$ in $\S^2$ using any of these functions.  To finish, it is enough to see that any of this isomorphisms produce an embedding of $\mm^b$ with boundary contained in  $\S^2\cap\{x=0\}$.
 \end{proof}

Finally, let us note that the ungluing procedure gives $\mm^b$ with exactly one more face that $\mm$: the external one. Furthermore, the internal faces of $\mm^b$ are in one to one correspondence with the faces of $\mm$. Thus, we can conclude the following lemma.
\begin{lemma}\label{l.p-angulations}
	For any $q\in \N$, the un-gluing function $\phi$ takes $m$-tree decorated $q$-angulations in to a pair of an $m$-tree and a $q$-angulation with boundary of size $2m$.
\end{lemma}

\begin{rem}\label{r.decoration}
	Let us also note that as vertex and faces are kept the same, except for the newly created external face, any face or vertex decoration commutes with the bijection.
\end{rem}

\subsubsection{From maps with simple boundary and trees to tree decorated maps} \label{s.gluing}
It is time to define the gluing function $\psi$ that takes a map with a boundary in $\MSB{f}{m}$ and a tree in $\PPT{m}$ and gives an $m$-tree decorated map and is the inverse of the ungluing function $\phi$. Informally, $\varphi$ should identify two oriented edges incident to the external face of the map using the relation given by the oriented edges in the tree.

Let $(\mm^b,\tt')$ be an element of $\MSB{f}{m}\times \PPT{m}$ and construct $(\mm,\tt)\in \DTM{f}{m}$, the value of $\psi((\mm^b,\tt'))$ as follows. Recall that the vertices of the external face of $\mm^b$ are indexed from $0$ to $2m-1$, and call $C$ the contour function of $\tt'$. Recall that $C$ induces an equivalent relationship on vertices via equation \eqref{e.EquivalenceC}, and define $V'$ as the set of equivalence classes.

Let us now construct $\mm$. The vertex set is made by the union of $V'$ with all vertices of $\mm^b$ that do not belong to the exterior face. The edge set of $\mm$ is constructed from those of $\mm^b$ in the following way. Let $(x,y)$ be an oriented edge of $\mm^b$, then the edge $(G(x),G(y))$ is in $\mm^b$, where
\begin{equation}
G(x):= \begin{cases}
[l]&\text{ if } x\in V',\\
x  &\text{else,} \end{cases}
\end{equation}
where $l$ is the label of $x$, and $[l]$ is the equivalence class of $l$ under the equivalence relationship defined by $C$.

Before defining the embedding, let us give some properties of the graph built.
\begin{itemize}
	\item If $(x,y)$ is an oriented edge that belongs to the boundary of $\mm^b$, its reverse $(y,x)$ is associated exactly to one other edge $(y',x')$ such that $(x',y')$ belongs to the boundary of $\mm^b$. In other words, $C$ induces a perfect matching of the edges in the boundary
	\item As the boundary of $\mm^b$ is simple, the image of the edges in the boundary of $\mm^b$ has the same tree structure as $\tt'$.  We define $\tt$ as this image.
\end{itemize} 

To finish the construction of $\mm$, we need to set the cyclical order of the edges around each vertex in the map. If $v$ is a vertex of $\mm$ that does not belong to $V'$, we set the order or the edges surrounding it as the order of $\mm^b$. In the case where $v\in V'$, we consider the order as the gluing of orders, following the corner identification around $v$ (see Figure (\ref{cornergluing})). Note that this also makes that the cycle order of $\tt$ is the same that the one of $\tt'$.
\begin{figure}[!h]
		\centering
		\scalebox{1}[1]{\includegraphics[scale=0.5]{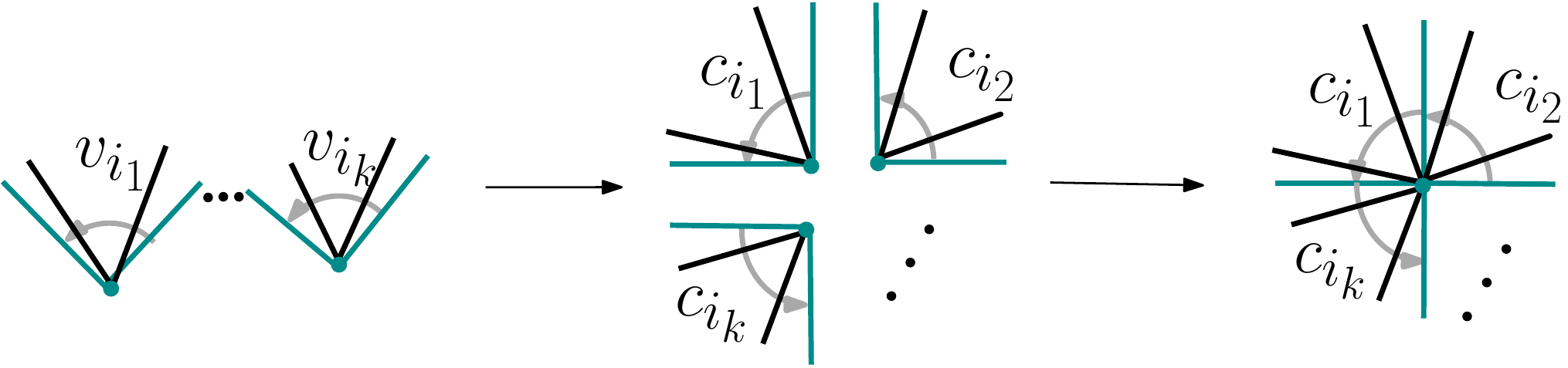}}
	% \vspace{-0.6cm}
	\caption{{\Small Gluing of corners and edge order around a vertex. For vertices $v_{i_1},v_{i_2},\dots,v_{i_k}$ with corners in the same class, we associate their order of edges and we glue these orders following the corner sequence of the contour function.}}
	\label{cornergluing}
\end{figure} 

To finish, let us actually prove that there exist an embedding that satisfy the given properties.
\begin{lemma}\label{l.embedding2}
	$(\mm,\tt)$ belongs to $\DTM{f}{m}$
\end{lemma}
\begin{proof}
	As $\tt$ is a subset of $\mm$, is a tree and has $m$ edges, we just need to prove that there exists an embedding that satisfies the order we imposed. This follows like in Lemma \ref{l.embedding1}, because of the existence of a orientation preserving isomorphism that goes from $\S^2\cap\{(x,y,z):x>0\}$ to $\S^2\backslash \tt$, and that takes any cyclically ordered prime ends of  $t$ in $\S^2\backslash \tt$ to elements of $\S^2\cap\{x=0\}$ keeping the same cyclical order.
\end{proof}

\begin{rem}
	Let us note that if we take $T$ to be a subset of the trees of size $m$, the gluing procedure produces a map decorated with a tree that belongs to $T$. This can be resume in the following lemma.
\end{rem}

\subsubsection{The gluing is actually a bijection} Let us now prove that the gluing and ungluing are inverse of each other.
\begin{prop}\label{p.bijection}
	For $f,m\in \N$, $\psi\circ \phi = Id$ in $\DTM{f}{m}$ and $\phi\circ \psi = Id$ in $\MSB{f}{m} \times \PPT{m}$.
\end{prop}
\begin{proof}
	Consider $\mm \in \DTM{f}{m}$, and let us show that $\psi(\phi(\mm))= \mm$. First note that the composition of both functions preserves the number of vertices. This is because, every vertex that does not belong to the decoration does not change by the transformations and the vertices on the decoration are separated by $\phi$ in corners and gathered by $\psi$ in exactly the same vertices of $\mm$.
	
	Let us now note that the edges are kept the same. This is because every edge without endpoints in the decoration remains after the transformations and edges with endpoints in the decoration again are unglued from the decoration by $\phi$ and glued after by $\psi$.
	
	Finally, we see that the orientation of the edges is kept. The reason for this is that for vertices that are not in the tree the orientation is kept the same, and for edges that hit in the tree, the orientation is carried by the tree that is the same in both cases.
\end{proof}
\begin{rem}
	Proposition \ref{p.bijection} can also be proven using that the isomorphisms taken in Lemma \ref{l.embedding1} and \ref{l.embedding2} may be chosen to be the inverse of each other.
\end{rem}

In fact, Lemma \ref{l.p-angulations} together with Proposition \ref{p.bijection} imply that the gluing-ungluing procedure is also a bijection when instead of fixing the number of faces, one fixes the number of edges in the map, which justifies \Cref{theo2}. They also imply that this is a bijection when restricted to $q$-angulations. 
 
\begin{prop}\label{p.q-angulations}
	$\phi$ is a bijection between $m$-tree decorated $q$-angulations with trees of $m$-edges and $q$-angulations with boundary of size $2m$.
\end{prop}

Furthermore, as the ungluing keeps the tree without any changes we obtain the following probabilistic result.

\begin{cor}\label{c.uniform}
The tree of a uniformly chosen $m$-tree decorated map (or  $q$-angulation) with $e+m$ edges is uniformly chosen between all of the trees of size $m$.
\end{cor}

\subsection{Extensions} In the following section, we discuss some extensions of the gluing procedure which allows us to make bijections between other combinatorial objects. 
\subsubsection{\tmb} In this subsection, we discuss what happens when gluing a tree to a map with a boundary that is bigger than the contour of the tree, and how this allows to create a dynamic gluing of the boundary.

Let us consider the gluing procedure described in Section \ref{s.gluing}, but now, let us glue a tree with a contour smaller than the boundary of the exterior face. Note that in this case the result is that the exterior face has shrunk but not disappeared. Furthermore, the glued part becomes a tree that shares only one vertex with the exterior face. This allows us to conclude the following proposition.

\begin{prop}\label{prop2}
	Let $f,m_1,m_2\in \N$. The set of tree-decorated map with a simple boundary, $\TMSB{f}{m_1}{m_2}$, is in bijection with the Cartesian product between trees and maps with  a simple boundary, $\PPT{m_2}\times\MSB{f}{m_1+2m_2}$.
\end{prop}

In this moment, let us present a fact about gluing that we believe may become useful in the future, especially when working with local limits. This fact is going to be presented in probabilistic terms, as it may be its main use.

Take $x\in \cro{0,m-1}$ and a contour function $C$, i.e., a Dyck path. Then, for every level $l\in \N$ smaller than the value of $C(x)$ there is a unique connected subset $\cro{\underline{m}(x,l),\overline{m}(x,l)}\ni x$ where $C(\cdot)-l$ is a Dyck path and
\begin{equation}
l=C(\underline{m}(x,l))=C(\overline{m}(x,l)).
%=C(\underline{m}(x,l))+1=C(\overline{m}(x,l))+1.	
\end{equation}
Note that the set of contour functions where this set is exactly $[\underline{m},\overline{m}]$ can just be reconstructed by the Cartesian product between a positive path going from $0$ to $l-1$ in $\underline{m}-1$ steps and a positive path going from $l-1$ to $0$ in $\overline{m}-1$ steps.

Looking from a tree point of view, this decomposition allows us to progressively recover all the vertices of the tree whose common ancestors with $x$ is at distance at least $l$ from the root. This allow us to define a progressive procedure of gluing trees, where at time $t \in \cro{0,l-1}$ we glue the sub-tree whose vertices live in $\cro{\underline{m}(x,l-t),\overline{m}(x,l-t)}$ with a given map having a simple boundary of size $2m$. 

From a probabilistic perspective these remarks imply the following result.
\begin{prop}\label{prop:dyn}
	Let $\mm^b$ be a map with a boundary of size $2m$ and $n$-vertices, chosen uniformly at random in $\MSB{f}{m}$ and $\tt$ be a tree of size $m$, also chosen uniformly at random in $\T{m}$. Additionally, take $x,l,\underline{m},\overline{m}\in \N$ with $x\in \cro{0,m-1}$. Call $\mm_t$ the map generated by gluing the sub-tree generated by the vertices $\cro{\underline{m}(x,l-t),\overline{m}(x,l-t)}$ with $\mm^b$. Conditionally on the event \[\{\cro{\underline{m}(x,l-t),\overline{m}(x,l-t)}=\cro{\underline{m},\overline{m}}\},\] the law $\mm_t$ is uniform over the $(\underline{m}-\overline{m})/2$-tree decorated maps with a simple boundary of size $m-(\underline{m}-\overline{m})$ and whose root is at cyclic distance $\min\{\underline{m},m-\overline{m}\}$ of the tree. Furthermore, under this conditioning,  the interior of $\mm_l$, the boundary of $\mm_t$, and the boundary $\mm_l$ minus the tree generated at time $t$ are all independent.
\end{prop}

In more probabilistic words, the maps $\mm_t$ works as a renewal submap for the gluing process of the tree.

\subsection{Gluing with a non-simple boundary}\label{sec:nonsimple}
To finish this section, we would like to discuss what happens when the boundary of the map we are gluing is not simple. The combinatorial objects that appear do not seem canonical, so we discuss with more emphasis the complication that make them arise.

We mostly focus on the case, where the boundary is bridgeless. In that case, we can make sense of a generalization of the gluing procedure. On the other hand, we will finish this section by explaining why it is not possible to make sense of the gluing procedure when the boundary has bridges.  

Let us start discussing the gluing between a map with a bridgeless boundary and a tree. Note that, in this case, the resulting glued ``map'' has a decoration that is not necessarily a tree, but only a submap (see Figure \ref{bubbles}). This generates two problems. The first one comes from the fact that two different gluing may give the same glued map with the same submap. This is not a central problem as it can be fixed by, instead of considering the decoration as  a submap, one considers it as a non-self crossing circuit, defined in \Cref{s.elementary def}. The circuit is just the image of the contour walk of the tree under the gluing, and the fact that the boundary was bridgeless, implies that the circuit only passes once by the same oriented-edge.
 
The second problem that arises in the gluing is the main one. The truth is that there are, in fact, two types of cycles that may appear in the circuit.  To describe them, let us, first, note that cycles can only appear in vertices that come from a non-simple vertex of the boundary. Thus, there are two possibilities for the image under the gluing of this vertex. Either, the tree glue it only with different points of the boundary. These generate harmless cycles. On the other hand, if the tree glue this vertex to itself, i.e., the vertex is associated to two corners belonging to the same equivalent class in the tree, then `wicked' cycles appear. That is to say, this vertex pinches down the sphere, and the map cannot longer be embedded in the sphere but only in a topological space with bubbles.
 
Let us explain, in a better way, what happens with the wicked points, for that it is useful to assume that the map with a boundary is already embedded in the sphere. Let $x$ be a wicked point. We know that if we retire $x$, the map with a boundary is left with two or more connected components of interior faces. Note that  two faces, $f_1,f_2$ made by edges in two different connected components can only be connected through a continuous path in the sphere that either pass through $x$ or pass through the exterior face of the map. After the gluing has been done, the exterior face disappears. Thus, any continuous path that goes from the image of $f_1$ to that of $f_2$ has to pass by $x$, as the tree gives no useful identification between these edges.  Thus, \textit{the resulting glued ``map'', is not a map} since it cannot be embedded in the sphere: removing the point $x$ results in a disconnection of the faces, which never happens in $\S^2$.

Additionally note that the bubbles are connected in an arborescent way, as they do not form cycles. This is because, each bubble is associated with a starting point and a subtree of the original tree, where two subtrees can only intersect in at most one vertex (See Figure \ref{bubbles}). Let us call bubble-maps any rooted graph that can be embedded in such a tree like topology.
 \begin{figure}[!h]
 	\begin{subfigure}{0.49\textwidth}
 		\centering
 		\scalebox{1}[1]{\includegraphics[scale=0.8]{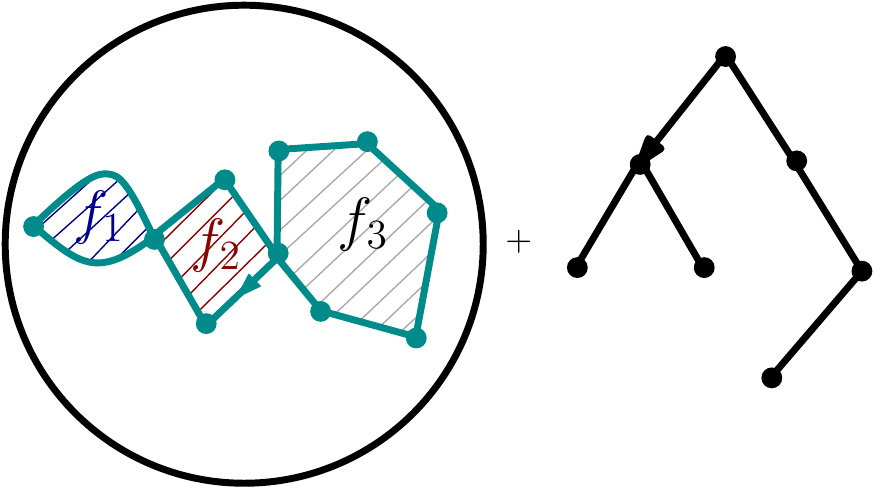}}
 	\end{subfigure}
 	\hspace{-0.75cm} $\longleftrightarrow$ \hspace{-0.25cm}
 	\begin{subfigure}{0.49\textwidth}
 		\centering
 		\scalebox{1}[1]{\includegraphics[scale=1.3]{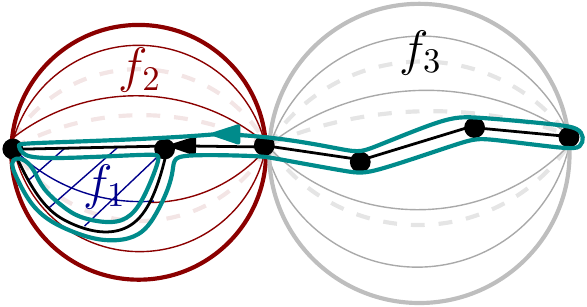}}
 	\end{subfigure}
 	% \vspace{-0.6cm}
 	\caption{{\Small Left: Bridgeless map with a non-simple boundary (interior faces are filled) and a tree. \\[2pt]
 			\noindent Right: Bubbles (3D plot) form by the gluing of a map with non-simple boundary and a tree. We chose to leave the scarf generated by the tree after the gluing (black), even though the decoration is the green circuit (oriented edges following the sense of the root edge from the root edge). }}
 	\label{bubbles}
 \end{figure}
 
 Furthermore, let us note that the image of the path given by the walker in the tree, generates a non-self crossing circuit that passes exactly two times by each edge. From this circuit, it is possible to recover the original tree structure. Let us describe an algorithm that allows to recover the contour function, $C$, of the tree by using the circuit. Start at $n=0$ and $C(0)=0$ from the root-vertex and following the root-edge, and iterate for each new edge of the circuit:
 \begin{itemize}
 	\item If the edge has been visited, set  $C(n+1)=C(n)-1$.
 	\item If this is the first time the edge has been visited and it visits a new vertex, set $C(n+1)=C(n)+1$.
 	\item If this is the first time the edge has been visited and it visits an already visited vertex, set $C(n+1)=C(n)+1$ and create a new non-self crossing circuit, where the visited vertex is duplicated. All edges visited before time $n$ (not including the one in this step), goes with the vertex going to the right, and all the others go with the vertex going to the left. At distance $\epsilon>0$ from this point, the graph is embedded homeomorphically as before. This can be done because the circuit is not self-crossing.
 	\item Set $n=n+1$.
 \end{itemize}
 Let us, now, define the image set of the gluing of a tree with a bridgeless map with a boundary. For $f,m\in \N$, we say that $(\mm,\cc)\in \PC{f}{m}$ if: $\mm$ is a bubble-map of $f$ faces, $\cc$ is a non-crossing circuit with length $2m$, going trough each edge twice, passing by every pinched point of the bubble-map and containing the root edge. Let us also define $\PCE{e}{m}$ as $\PC{f}{m}$, where in the place of $f$ faces, we consider $e+m$ edges.
 
 Let us summarize the discussion.
 \begin{prop}\label{prop:nonsimple}
 	When one glues a tree $\tt' \in\T{m}$ with map with a (non-simple) bridgeless boundary $\mm^b\in \MB{f}{m}$, one only obtains a map if there is no vertex where the boundary of $\mm^b$ is not simple and that is identified by $\tt'$.
 	
 	More importantly, the  gluing is a bijection between $ \PPT{m}\times \MB{f}{m}$ and $\PC{f}{m}$.
 \end{prop}
\begin{proof}
	Given the above discussion, we only need to show that we can perform the ungluing. Let us note that the tree can be recovered from the algorithm described above. To recover the map with a boundary, one just needs to duplicate the edges crossed as described in Figure \ref{dectoboundandtree1}. Now, we only have to explain what needs to be done close to the pinch points. In those points, one just needs to locally modify the underlying space so that the point have thickness and all the newly added points belong to the internal face.
\end{proof} 

\begin{rem}\label{rem:non-simple}
	Again, the bijection does not modify the degree of inner faces, so the same statement is valid when we restrict our attention to bubble $q$-angulations.
\end{rem}

As discussed in the introduction, this bijection is not that useful in the case of maps with a given number of faces. However, it can be adapted to families of maps with a given number of edges. 

 \begin{prop}\label{prop:nonsimpleedge}
	The set $\PCE{e}{m}$ is in bijection with the cartesian product of $\T{m}$ with the set of map with (non-simple) bridgeless boundary with $e$ internal edges and boundary of size $2m$.
\end{prop}

Let us finish this section by describing what happens when the boundary have bridges. In this case, we do not only have problems with identification of vertices, but also with the identification of edges. In particular, there may be two bridges in the boundary having oriented edges identified in the tree. This makes that the circuit we create passes at least two times by the image of that oriented edge, and thus, it  makes it impossible to reverse the gluing.

 \begin{figure}[!h]
	\begin{subfigure}{0.49\textwidth}
		\centering
		\scalebox{1}[1]{\includegraphics[scale=0.8]{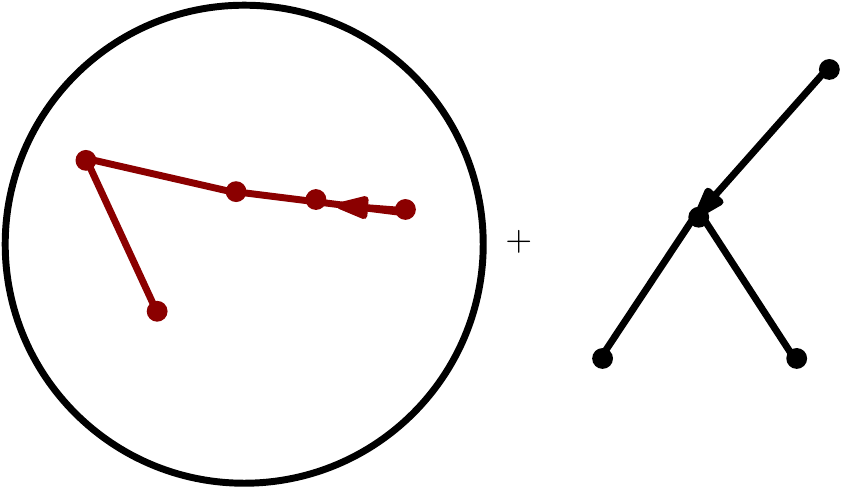}}
	\end{subfigure}
	\hspace{-0.75cm} $\longleftrightarrow$ \hspace{-0.25cm}
	\begin{subfigure}{0.49\textwidth}
		\centering
		\scalebox{1}[1]{\includegraphics[scale=1.5]{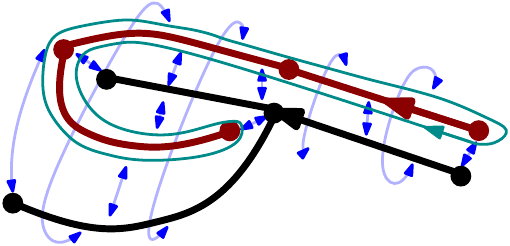}}
	\end{subfigure}
	% \vspace{-0.6cm}
	\caption{{\Small Left: Bridge map with a non-simple boundary (red tree) and a tree to glue. \\[2pt]
			\noindent Right: The gluing of the objects on the left. As in \cref{bubbles} the green part represents the circuit (each edge in the red tree appears twice as we follow the root edge on the boundary of the root-face). We put blue lines for identifications made by the gluing. The edges of the black tree that do not contain the root-edge generate only one edge $e_1$, after the gluing, what makes the green circuit visit four times $e_1$.}}
	\label{bubbles2}
\end{figure}
 
\section{Countings}\label{count}
In this section, we discuss how the bijections presented are translated in to counting formulae. Before stating the results, we need to introduce a re-rooting procedure, which  allows us to obtain formulas for decorated maps whose root is not necessarily in the decoration.

For convenience, in this section, we write $(q)$ for a family of maps to consider the case where all the inner faces are $q$-angulations. 
\subsection{Re-rooting procedure}\label{rerooting}
In this section, we explain how the re-rooting procedure works for $r$-forest $q$-angulations. The main result of this section is the fact that
\begin{equation}\label{e.rooting}
|\DFM{f}{m_1,..,m_r}(q)|m(f,q)=|\DFMA{f}{m_1,..,m_r}(q)| r!\left( \prod_{i=1}^{r} 2m_i\right)\left(\prod_k c_k! \right)^{-1} ,
\end{equation}
where $m(f,q)$ is the amount of edges of a $q$-angulation with $f$ faces, and  $c_k$ counts the multiplicity of the elements of identical $m_i$, i.e., 
\begin{equation*}
c_k:=|\{ i\in\cro{1,r}: m_i=k  \}|.
\end{equation*}

Let us now deduce Equation \eqref{e.rooting}. It just follows from counting a slightly bigger set. Let us call \begin{equation*}
\DFMAG{f}{m_1,..,m_r}(q):= \left \{(\mm,\tt_1,...,\tt_r,\vec{e}):(\mm,\tt_1,...,\tt_r)\in \DFM{f}{m_1,..,m_r}(q) \text { and $\vec{e}$ is an oriented edge of $\mm$}\right \}. 
\end{equation*}
 In other words, an element of $\DFMAG{f}{m_1,..,m_r}(q)$ is made of an element of $\DFM{f}{m_1,..,m_r}(q)$ together with an additional marked edge. As any automorphism of a map that fix one oriented edge, fixes all of them, we have that we can compute the cardinal of $\DFMAG{f}{m_1,..,m_r}(q)$ in two ways. The first one is by multipling the amount of possible values of $\vec{e}$ by the cardinal of $\DFM{f}{m_1,..,m_r}(q)$, this gives the left-hand side of \eqref{e.rooting}. The other one, is to multiply the cardinal of $\DFMA{f}{m_1,..,m_r}(q)$ by the possible position of the root of each tree, and then by all possible orders of these $r$ trees, this gives the right-hand side of \eqref{e.rooting}.

\begin{rem}
	It is important to remark that we can do this procedure as soon as we work with two types of rootings of the same family of maps. The justification is subtle since if one unroots a given map, certain symmetries may appear, breaking down the argument. Instead, when distinguishing more than one edge symmetries do not appear and this type of identities follow.
\end{rem}

\subsection{Counting relation between maps with a boundary and maps with a simple boundary}
\label{ssec:sbound}

The main interest of this section is to compute the generating functions of the maps with a simple boundary, as they appear in the bijection presented in Section \ref{s.b}. To do it, we are going to adapt the technique introduced in \cite{BG09} that were used to link the generating function of quadrangulations with a boundary to that of quadrangulations  with a simple boundary.

Let us start by noting that the set of maps with a simple boundary and $f$ faces is infinite. Instead, one needs to specify the number of edges and the size of the boundary. 

We define the following generating functions
\begin{align*}
&\BBB(x,y)= \sum_{e=0}^\infty\sum_{p=0}^\infty\bbb_{e,p}x^ey^p,\\
&\SSS(x,y)=\sum_{e=0}^\infty\sum_{p=0}^\infty \sss_{e,p}x^ey^p,
\end{align*}
where $\bbb_{e,p}$, resp. $\sss_{e,p}$, is the number of general maps with $e$ edges and $p$ edges on the boundary, where the boundary is simple for the case of $\sss_{e,p}$.

Similar to \cite{BG09}, we obtain the following identity relying $\SSS$ and $\BBB$.
\begin{align}\label{eq}
\SSS(x,y\BBB(x,y)) &= \BBB(x,y).
\end{align}
Let us sketch the justification: a general map with a boundary can be seen as the component with simple boundary where the root-edge lies gathered, with general maps with a boundary hanging from the boundary vertices. Therefore, for each edge in the boundary of a general map with simple boundary we count a weight $y\BBB(x,y)$ to recover all maps with a boundary. The weight $y\BBB(x,y)$ is associated to the weight of the edge in the boundary and the map hanging from the tail of this edge, when considering oriented following the contour of the simple boundary in the sense of the root-edge.

The only difference with \cite{BG09}, is that the external boundary may have any possible length. For the case of quadrangulations the external boundary has to be of even length; this forces in \cite{BG09} $Y=yW(x,y)^2$, with $W$ the generating function with coefficients $w_{e,p}$ counting the number of quadrangulations with $e$ edges and boundary of size $2p$. When applying this technique, depending on the family of maps under study, it is important to take into account this type of restriction to obtain non-degenerate cases.

Now we use this change of variable to discover $\SSS$. To start, it is well known (see, for example, \cite[VII.8.2]{FS09}) that $\BBB(x,y)$ satisfies
\begin{align}\label{eq:2}
	\BBB(x,y) = 1+y^2x\BBB(x,y)^2 + \frac{xy}{1-y}\left( \BBB(x,1)-y\BBB(x,y) \right),
\end{align}
where $\BBB(x,1)$ is the counting formula for general maps, with the following explicit form
\begin{align}\label{eq:3}
	\BBB(x,1) = \sum_{e=0}^\infty \frac{2\cdot 3^e}{(e+1)(e+2)}\binom{2e}{e} x^e = -\frac{1}{54x^2}\left( 1-18x-(1-12x)^{3/2} \right).
\end{align}
Now, turning into general maps with simple boundary, make the change of variable $z=y\BBB(x,y)$ in \cref{eq:2} to obtain
\begin{align}
	\SSS(x,z) = 1+xz^2 + \frac{xz}{\SSS(x,z)-z}\left( \BBB(x,1)- z \right).
\end{align}
This gives a quadratic equation for $\SSS$, obtaining the desired function $\SSS$ the only possible solution of this equation with positive coefficients
\[
	\SSS(x,z)=  \frac{1}{2}\left(1+z+xz^2-\sqrt{(-xz^2+z+1)^2-\frac{2z(1+36x-(1-12x)^{3/2})}  {27x}}\right),
\]
Its first coefficients are shown in the expansion $\SSS(x,z)= 1+xz+2x^2z+xz^2+9x^3z+x^2z^2+54x^4z+5x^3z^2+378x^5z+32x^4z+x^3z^3+...$.

This generating function codes the number of general maps with a simple boundary, the coefficients may be recover from it doing a sequence of derivations. Nevertheless, they do not have a closed form to our knowledge.
\color{black}
\subsection{Counting results}\label{s.counting}
Before presenting the results, let us start by recalling that the number of (rooted) trees on $n$ edges are counted by Catalan numbers $\{C_n\}_{n\in \N}$ (see, for example, Section 1 in \cite{Ber07})
\begin{equation}
|\T{n}| = \C_n := \frac{1}{n+1}\binom{2n}{n}.
\end{equation} 
\subsubsection{Tree and forest decorated maps}Let us now obtain the formulas that directly come from the bijection for tree-decorated maps. The counting formulas obtained for maps come from \cite{Kr07} for the case of triangulations with multiples simple boundaries, and from \cite{BerF18} on the case of quadrangulations. 
  
\begin{cor}\label{cor1}
	Set $m = m_1+m_2$. The number of tree-decorated triangulations with boundary of size $m_1$ and a tree of size $m_2$ is
	\begin{equation}\label{e.tree simple boundary tri}
	|\TMSB{f}{m_1}{m_2}(3)| =  2^{f-2m}\frac{(3f/2+m-2)!!}{(f/2-m+1)!(f/2+3m)!!}\frac{2m}{2m_2+1}\binom{4m}{2m}
	\binom{2m_2}{m_2},
	\end{equation}
	where $n!!$ stands from double factorial (or semifactorial).
	The cardinal of tree-decorated quadrangulations with boundary of size $m_1$ and a tree of size $m_2$ is
	\begin{equation}\label{e.tree simple boundary quad}
	|\TMSB{f}{m_1}{m_2}(4)| = 3^{f-m}\frac{(2f+m-1)!}{(f+2m)! (f-m+1)!}  \frac{2m}{m_2+1}\binom{3m}{m}\binom{2m_2}{m_2}
	\end{equation}
\end{cor}
\begin{proof}
	This is obtained from Proposition \ref{prop2}, and the comment above Lemma \ref{l.p-angulations}. The formula for the number of trees with $m_2$ edges and Theorem 1 of \cite{Kr07} and Section 2.2 of \cite{BG09} respectively.
\end{proof}

We can now discuss the results for trees.
\begin{cor}\label{cor2}
	The cardinal of r-forest decorated triangulations, with trees of size $m_1,\;m_2,\dots,m_r$, is given by
	\begin{equation}\label{e.forest tri}
		|\DFMA{f}{m_1,..,m_r}(3)| = 2^{f-2m}3f\frac{(3f/2+m-2)!!}{(f/2-m+2-r)!(f/2+3m)!! } \frac{r!}{\prod_{i\in \N} c_k!}\prod_{i=1}^r \frac{1}{m_i+1}\binom{4m_i}{2m_i,m_i,m_i} 
	\end{equation}
	and r-forest decorated quadrangulations, with trees of size $m_1,\;m_2,\dots,m_r$, is given by: 
%	\frodo{los hoyos son de perimetro $2 a_i$}
	\begin{equation}\label{e.forest quad}
		|\DFMA{f}{m_1,..,m_r}(4)| = 3^{f-m}4f\frac{(2f+m-1)!}{(f+2m)! (f-m+2-r)!}\frac{r!}{\prod_{i\in \N} c_k!} \prod_{i=1}^r\frac{1}{m_i+1}\binom{3m_i}{m_i,m_i,m_i}  
	\end{equation}
\end{cor}
\begin{proof}
	This is obtained from \Cref{prop2}, the comment above Lemma \ref{l.p-angulations}, the formula for the number of trees with $m_i$ edges and the results of \cite[Theorem 1]{Kr07} and \cite[Theorem 1.2]{BerF18} respectively, giving that
	\begin{equation}\label{e.forest tri2}
	|\DFM{f}{m_1,..,m_r}(3)| = \frac{2^{f-2m}(3f/2+m-2)!!}{(f/2-m+2-r)!(f/2+3m)!!}\prod_{i=1}^r \frac{2m_i}{m_i+1}\binom{4m_i}{2m_i,m_i,m_i} 
	\end{equation}
	and 
	%	\frodo{los hoyos son de perimetro $2 a_i$}
	\begin{equation}\label{e.forest quad2}
	|\DFM{f}{m_1,..,m_r}(4)| = 3^{f-m}\frac{(2f+m-1)!}{(f+2m)! (f-m+2-r)!} \prod_{i=1}^r\frac{2m_i}{m_i+1}\binom{3m_i}{m_i,m_i,m_i}.  
 	\end{equation}
	We conclude using the re-rooting procedure condensed in \eqref{e.rooting}.
\end{proof}

\begin{rem}
	Let us note that Corollary \ref{c.counting} can be obtained from \Cref{cor2} using $r=1$.
\end{rem}

Similar counting formulas can be obtained for triangulations of girth bigger than 2 (loopless triangulations) and 3; and loopless quadrangulations (see \cite{BerF18}). In \Cref{cor2}, one could also consider a generalization for ``tree-decorated maps with boundaries'' an analog of the tree-decorated maps with a boundary for multiples boundaries. We hope that the interested reader will have no problem in finding the formula.

As discussed in \Cref{ssec:sbound}, the number of general maps with a simple boundary does not have a closed formula to our knowledge, still it is possible to obtain the number of general maps decorated by a tree once extracted the coefficients $\sss_{e,p}$ with $p$ even. More formally, the number of general maps decorated by a tree of size $m$ and with $e$ edges is given by $\C_{m}\sss_{e+m,2m} $, which is justified by \Cref{theo2}.

The cardinal of Bubble-maps can be obtained from \Cref{prop:nonsimpleedge}.
\begin{cor}\label{cor3}
	The cardinal of non-crossing circuit decorated bubble-maps $(\mm,\cc)$ with $e+m$ edges decorated by a circuit of size $2m$ and with root in an oriented edge of the map $\mm$ is
	\begin{equation}
	|\PCAE{e+m}{m}| = 3^{e}\frac{(2e+2m-1)!}{e! (e+2m+1)!} \frac{2(e+m)}{m+1}\binom{4m}{2m,m,m}
	\end{equation}
\end{cor}

\begin{proof}
	This is obtained from \Cref{prop:nonsimpleedge}, the formula for the number of trees with $m$ edges, the results of \cite[Section 2.2]{FG15} together with \cite[Section 2.2]{BG09} and a type of re-rooting procedure, similar to the one in \Cref{rerooting}.
\end{proof}

We leave to the reader the computation of the formula of maps decorated by a special type of trees, as for example trees with prescribed degree distribution, see for example \cite{Stanley}.
\subsubsection{Spanning tree-decorated maps}\label{s.stree}
In this subsection, we discuss the consequence of our result for spanning tree-decorated maps. In this case, the counting formula of spanning tree-decorated maps for general face degree was given by Mullin \cite{Mul67}, where the root does not necessarily belongs to the tree.
It states that spanning tree-decorated maps with $e$ edges are in correspondence with a pair of trees with $e$ and $e+1$ edges, i.e. is counted by $C_{e}C_{e+1}$. Later Bernardi \cite{WL72,CDV86,Ber07} gave a bijective proof of this result.
 
We denote by $\DSTM{f}$, resp. $\DSTMA{f}$, the set of spanning tree-decorated maps on $f$ faces with one root-edge in the tree, resp. in the map.
Note that in the case of triangulations we obtain that for $f$ faces the number of vertices is $2+f/2$ and the number of edges is $3/2f$, therefore, the number of spanning-tree decorated triangulations are
\begin{align}
|\DSTMA{f}(3)|&=|\DTMA{f}{f/2+1}(3)| = \frac{12f}{(f+4)(f+2)^2}\binom{2f}{f,f/2,f/2}.
\end{align}
In the case of quadrangulations, the condition of having $f$ faces implies that it has $f+2$ vertices and $2f$ edges. Therefore, we obtain the counting formula for spanning-tree decorated quadrangulations.
\begin{align}
|\DSTM{f}(4)|=|\DTM{f}{f+1}(4)| = \frac{2}{(f+1)(f+2)}\binom{3f}{f,f,f}\label{eq:catext}\\
|\DSTMA{f}(4)|=|\DTMA{f}{f+1}(4)| = \frac{4f}{(f+1)^2(f+2)}\binom{3f}{f,f,f}.
\end{align}

Here we recover the results ontained from the Walsh and Lehman's Bijection \cite{WL72}. In fact, the dual of $\DSTM{f}(3)$ are the spanning tree-decorated $3$-regular maps with $f$ vertices, and that of $\DSTM{f}(4)$ are the spanning tree-decorated $4$-regular maps with $f$ vertices. These can be counted as appears in Section 6.2 of \cite{BM11}, where they explicit the formula for more general families of maps, given that (dual) trees with prescribed degree distribution are easily counted.

\begin{rem}
	Notice that we kept the expressions for these spanning-tree decorated quadrangulations, with root in the tree and in the map. We want to point out that the right side of \eqref{eq:catext} looks like a possible generalization of the Catalan numbers. More rigorously, for $n,m\in \N$, $m\geq 1$ define:
	\[
	\C_{m,n} = m!\left(\prod_{i=1}^{m}\frac{1}{n+i}\right) \binom{(m+1)n}{\underbrace{n,n,\dots,n}_{m+1 \text{ times}}} = \binom{m+n}{n}^{-1}\binom{(m+1)n}{\underbrace{n,n,\dots,n}_{m+1 \text{ times}}}
	\]
	When $m=1$, we recover the Catalan numbers and, for this definition, $\C_{2,f}$ counts $|\DSTM{f}(4)|$. To our knowledge, this extension has not been defined so far and it does not appear in the OEIS \footnote{$4-th$ march 2019 update: it has been added for $m=2,3$.}.
\end{rem}
	
	From the definition it is not clear to see that $\C_{m,n}$ is, in fact, an integer. Luckily for us, Vincent Jugé found an elegant, analytical proof of this fact that we present in the following proposition.	
	\begin{prop}
		For all $n,\; m\in \N$ and $m\geq 1$, $\C_{m,n}$ is integer.
	\end{prop}
	\begin{proof}
		Define $\nu_p(k)$ as the largest power of $p$ prime that divides $k\in \N$. Recall that by Legendre's formula 		
		\[
		\nu_p(k!)= \sum_{i=1}^\infty \left\lfloor \frac{k}{p^i}  \right\rfloor.
		\]
		Thanks to this, we can calculate the maximal power of $p$ prime that divides $\C_{m,n}$
		\begin{align*}
		&\nu_p((m+1)n!)-\nu_p((n!)^{m+1})-\nu_p((n+m)!)+\nu_p(n!)+\nu_p(m!)\\
		&\hspace{0.2\textwidth}= \sum_{i=1}^\infty 
		\underbrace{\left\lfloor \frac{(m+1)n}{p^i} \right\rfloor
			-(m+1)\left\lfloor \frac{n}{p^i}\right\rfloor}_{=:(1)\geq 0}
		\underbrace{-\left\lfloor \frac{m+n}{p^i}\right\rfloor
			+\left\lfloor \frac{n}{p^i}\right\rfloor
			+\left\lfloor \frac{m}{p^i}\right\rfloor}_{=:(2)\geq -1}
		\end{align*}
		Note that to conclude we just need to show that each term in the summation is bigger than or equal to zero. To do this notice that we just need to show that it cannot happen simultaneously $(1)=0$ and $(2)=-1$.
		
		Assume that this is the case and write
		\begin{align*}
		&n=k p^i + l_n,\\
		&m= k' p^i +l_m,
		\end{align*}
		with $0\leq l_n, l_m <p^i$ and $l_m<m$. Note that the fact that $(1)=0$ implies that \begin{equation}
		\label{e. ln}l_n (m+1)< p^i.
		\end{equation}
		 Futhermore, the fact that $(2)=-1$ implies that $l_n\neq 0$ and \begin{equation*}
		 l_n+l_m \geq p^i.
		 \end{equation*}
		 Together with \eqref{e. ln} this implies that $m l_n< l_m\leq m$ which implies that $l_n=0$ so we have a contradiction.
	\end{proof}

\subsubsection{Spanning r-forest decorated maps} The case of forest-decorated maps is obtained from \Cref{cor2}. As before, a triangulation with $f$ faces has $2+f/2$ vertices, which has to be equal to $m+r = \sum_{i=1}^rm_i+r$, the number of vertices covered by the forest. Thus,  
\begin{equation}
|\DSFMA{f}{m_1,..,m_r}(3)| =4^{r-2}3f\frac{(2f+2-r)!!}{(2f+6-3r)!!}\frac{r!}{\prod_{i\in \N} c_k!}\prod_{i=1}^r\frac{1}{m_i+1}\binom{4m_i}{2m_i,m_i,m_i} 
\end{equation}
Again as before, a quadrangulation with $f$ faces has $f+2$ vertices, which has to be equal to $m+r$, the number of vertices cover by the forest. Thus,
\begin{equation}
|\DSFMA{f}{m_1,..,m_r}(4)| = 3^{r-2}4f\frac{(3f-r+1)!}{(3f-2r+4)!} \frac{r!}{\prod_{i\in \N} c_k!}\prod_{i=1}^r \frac{1}{m_i+1}\binom{3m_i}{m_i,m_i,m_i}
\end{equation}
\begin{rem}
	From this formula it is also possible to deduce \Cref{cor:sptree}.
\end{rem} 
\bibliographystyle{alpha} 
\bibliography{biblio}

\begin{thebibliography}{BDFG04}

\bibitem[BBG11]{BBG12}
Gaetan Borot, J{\'{e}}r{\'{e}}mie Bouttier, and Emmanuel Guitter.
\newblock A recursive approach to the {O (n)} model on random maps via nested
  loops.
\newblock {\em Journal of Physics A: Mathematical and Theoretical},
  45(4):045002, 2011.

\bibitem[BDFG04]{BG}
J{\'e}r{\'e}mie Bouttier, Philippe Di~Francesco, and Emmanuel Guitter.
\newblock Planar maps as labeled mobiles.
\newblock {\em The electronic journal of combinatorics}, 11(1):69, 2004.

\bibitem[Ber07]{Ber07}
Olivier Bernardi.
\newblock Bijective counting of tree-rooted maps and shuffles of parenthesis
  systems.
\newblock {\em The electronic journal of combinatorics}, 14(1):9, 2007.

\bibitem[Bet15]{BetQ}
J{\'e}r{\'e}mie Bettinelli.
\newblock Scaling limit of random planar quadrangulations with a boundary.
\newblock 51(2):432--477, 2015.

\bibitem[BF18]{BerF18}
Olivier Bernardi and {\'E}ric Fusy.
\newblock Bijections for planar maps with boundaries.
\newblock {\em Journal of Combinatorial Theory, Series A}, 158:176--227, 2018.

\bibitem[BG09]{BG09}
J{\'e}r{\'e}mie Bouttier and Emmanuel Guitter.
\newblock Distance statistics in quadrangulations with a boundary, or with a
  self-avoiding loop.
\newblock {\em Journal of Physics A: Mathematical and Theoretical},
  42(46):465208, 2009.

\bibitem[BM11]{BM11}
Mireille Bousquet-M{\'e}lou.
\newblock Counting planar maps, coloured or uncoloured.
\newblock In {\em 23rd British Combinatorial Conference}, volume 392, pages
  1--50. 2011.

\bibitem[BM17]{BetM}
J{\'e}r{\'e}mie Bettinelli and Gr{\'e}gory Miermont.
\newblock Compact brownian surfaces {I}: Brownian disks.
\newblock {\em Probability Theory and Related Fields}, 167(3-4):555--614, 2017.

\bibitem[CC16]{CarCur}
Alessandra Caraceni and Nicolas Curien.
\newblock Self-avoiding walks on the uipq.
\newblock {\em arXiv preprint arXiv:1609.00245}, 2016.

\bibitem[CDV86]{CDV86}
Robert Cori, Serge Dulucq, and G{\'e}rard Viennot.
\newblock Shuffle of parenthesis systems and baxter permutations.
\newblock {\em Journal of Combinatorial Theory, Series A}, 43(1):1--22, 1986.

\bibitem[DG18]{DGExpB}
Jian Ding and Ewain Gwynne.
\newblock The fractal dimension of {L}iouville quantum gravity: universality,
  monotonicity, and bounds.
\newblock {\em arXiv preprint arXiv:1807.01072}, 2018.

\bibitem[DK88]{DK88}
Bertrand Duplantier and Ivan Kostov.
\newblock Conformal spectra of polymers on a random surface.
\newblock {\em Physical review letters}, 61(13):1433, 1988.

\bibitem[Edm60]{Ed60}
Jack Edmonds.
\newblock A combinatorial representation of polyhedral surfaces.
\newblock {\em Notices of the American Mathematical Society}, 7, 1960.

\bibitem[FG15]{FG15}
{\'E}ric Fusy and Emmanuel Guitter.
\newblock Comparing two statistical ensembles of quadrangulations: a continued
  fraction approach.
\newblock {\em arXiv preprint arXiv:1507.04538}, 2015.

\bibitem[FS09]{FS09}
Philippe Flajolet and Robert Sedgewick.
\newblock {\em Analytic combinatorics}.
\newblock Cambridge University Press, Cambridge, 2009.

\bibitem[FS19]{FS}
Luis Fredes and Avelio Sep{\'u}lveda.
\newblock Limits of tree-decorated maps.
\newblock {\em To appear}, 2019.

\bibitem[GHS17]{GHS}
Ewain Gwynne, Nina Holden, and Xin Sun.
\newblock A mating-of-trees approach to graph distances in random planar maps.
\newblock {\em arXiv preprint arXiv:1711.00723}, 2017.

\bibitem[GJ04]{GJ04}
Ian Goulden and David Jackson.
\newblock {\em Combinatorial enumeration}.
\newblock Courier Corporation, 2004.

\bibitem[GM16]{GM2}
Ewain Gwynne and Jason Miller.
\newblock Convergence of the self-avoiding walk on random quadrangulations to
  {SLE$_{8/3}$} on {$\sqrt{8/3}$ }-liouville quantum gravity.
\newblock {\em arXiv preprint arXiv:1608.00956}, 2016.

\bibitem[GM19]{GM}
Ewain Gwynne and Jason Miller.
\newblock Convergence of the free {B}oltzmann quadrangulation with simple
  boundary to the brownian disk.
\newblock {\em To appear in Annales de l'institut Henri Poincar{\'e}}, 2019.

\bibitem[JS04]{JS}
Wolfhard Janke and Adriaan Schakel.
\newblock Geometrical vs. fortuin--kasteleyn clusters in the two-dimensional
  q-state potts model.
\newblock {\em Nuclear Physics B}, 700(1-3):385--406, 2004.

\bibitem[Kri07]{Kr07}
Maxim Krikun.
\newblock Explicit enumeration of triangulations with multiple boundaries.
\newblock {\em Electronic Journal of Combinatorics}, 14(1):Research Paper 61,
  14, 2007.

\bibitem[LG05]{LG05}
Jean-Fran{\c{c}}ois Le~Gall.
\newblock Random trees and applications.
\newblock {\em Probability surveys}, 2:245--311, 2005.

\bibitem[LG13]{LeGall}
Jean-Fran{\c{c}}ois Le~Gall.
\newblock Uniqueness and universality of the brownian map.
\newblock {\em The Annals of Probability}, 41(4):2880--2960, 2013.

\bibitem[Mie13]{Mier}
Gr{\'e}gory Miermont.
\newblock The brownian map is the scaling limit of uniform random plane
  quadrangulations.
\newblock {\em Acta mathematica}, 210(2):319--401, 2013.

\bibitem[MM06]{MM}
Jean-Fran{\c{c}}ois Marckert and Abdelkader Mokkadem.
\newblock Limit of normalized quadrangulations: the brownian map.
\newblock {\em The Annals of Probability}, 34(6):2144--2202, 2006.

\bibitem[MR72]{MR72}
Arak Mathai and Pushpa Rathie.
\newblock Enumeration of almost cubic maps.
\newblock {\em Journal of Combinatorial Theory, Series B}, 13(1):83--90, 1972.

\bibitem[Mul67]{Mul67}
Ronald Mullin.
\newblock On the enumeration of tree-rooted maps.
\newblock {\em Canadian Journal of Mathematics}, 19:174--183, 1967.

\bibitem[Pom13]{Pom}
Christian Pommerenke.
\newblock {\em Boundary behaviour of conformal maps}, volume 299.
\newblock Springer Science \& Business Media, 2013.

\bibitem[Sch98]{SCH98}
Gilles Schaeffer.
\newblock {\em Conjugaison d'arbres et cartes combinatoires al{\'e}atoires}.
\newblock PhD thesis, Universit\'{e} Bordeaux 1, 1998.

\bibitem[She16]{She}
Scott Sheffield.
\newblock Quantum gravity and inventory accumulation.
\newblock {\em The Annals of Probability}, 44(6):3804--3848, 2016.

\bibitem[Sta99]{Stanley}
Richard~P. Stanley.
\newblock {\em Enumerative combinatorics. {V}ol. 2}, volume~62 of {\em
  Cambridge Studies in Advanced Mathematics}.
\newblock Cambridge University Press, Cambridge, 1999.
\newblock With a foreword by Gian-Carlo Rota and appendix 1 by Sergey Fomin.

\bibitem[Tut62]{Tut62}
William Tutte.
\newblock A new branch of enumerative graph theory.
\newblock {\em Bulletin of the American Mathematical Society}, 68(5):500--504,
  1962.

\bibitem[WL72]{WL72}
Timothy Walsh and Alfred Lehman.
\newblock Counting rooted maps by genus ii.
\newblock {\em Journal of Combinatorial Theory, Series B}, 13(2):122--141,
  1972.

\end{thebibliography}
\end{document}